\documentclass[11pt]{amsart}
\usepackage[margin=1in,letterpaper,portrait]{geometry}
\usepackage{latexsym,amssymb,verbatim,multirow,graphicx,epsfig,enumerate,paralist,hyperref,mathtools,mathrsfs}
\usepackage{graphbox}
\usepackage{subcaption}
\usepackage{xcolor}
\usepackage[normalem]{ulem}
\usepackage{tikz}
\usepackage{bm}
\usepackage{float}
\usepackage{diagbox}
\usepackage{booktabs}

\usepackage{multirow}

\usepackage{thm-restate}

\theoremstyle{plain}
   \newtheorem{theorem}{Theorem}[section]
   \newtheorem{proposition}[theorem]{Proposition}
   
   \newtheorem{lemma}[theorem]{Lemma}
   \newtheorem{corollary}[theorem]{Corollary}

   \newtheorem*{theorem*}{Theorem}
\theoremstyle{definition}
   
   \newtheorem{definition}[theorem]{Definition}

   \newtheorem{remark}[theorem]{Remark}

\numberwithin{equation}{section}

\newcommand{\op}[1]{\operatorname{#1}}
\newcommand{\defn}[1]{{\color{blue} \it {#1}}}

\newcommand{\AAA}{\mathcal{A}}
\newcommand{\CCC}{\mathcal{C}}
\newcommand{\LLL}{\mathcal{L}}
\newcommand{\PPP}{\mathcal{P}}
\newcommand{\QQQ}{\mathcal{Q}}

\newcommand{\far}{\mathcal{F}}
\newcommand{\nfar}{\mathcal{NF}}
\newcommand{\nfw}{\mathcal{G}}
\newcommand{\nfwse}{\nfw_{\mathrm{se}}}

\newcommand{\spn}{\mathrm{sp}}

\begin{document}

\title{Counting nearest faraway flats for Coxeter chambers}
\author{Theo Douvropoulos}
\maketitle


\begin{abstract}
In a finite Coxeter group $W$ and with two given conjugacy classes of parabolic subgroups $[X]$ and $[Y]$, we count those parabolic subgroups of $W$ in $[Y]$ that are full support, while simultaneously being \emph{simple} extensions (i.e., extensions by a single reflection) of some \emph{standard} parabolic subgroup of $W$ in $[X]$. The enumeration is given by a product formula that depends only on the two parabolic types. Our derivation is case-free and combines a geometric interpretation of the ``full support" property with a double counting argument involving Crapo's beta invariant. As a corollary, this approach gives the first case-free proof of Chapoton's formula for the number of reflections of full support in a real reflection group $W$.
\end{abstract}

\section{Introduction}
A long time ago in Waterloo, Ontario, Crapo \cite{crapo} introduced a numerical invariant for matroids which has since been known as the \emph{beta invariant}. When the matroid comes from an essential, central hyperplane arrangement $\AAA$, the beta invariant \defn{$\beta(\AAA)$} can be defined in terms of the characteristic polynomial $\chi(\AAA,t)$ as follows:
\begin{equation}
    \beta(\AAA):=(-1)^{\op{rank}(\AAA)-1}\cdot \dfrac{\op{d}\chi(\AAA,t)}{\op{d}t}\Big|_{t=1}.\label{eq:intro:defn:beta}
\end{equation}

Hyperplane arrangements and their invariants have been used widely to encode and study combinatorial objects. This has been very successful in the setting of Coxeter and Coxeter-Catalan combinatorics where remarkable enumerative formulas \cite{athan_enum,fomin_read,athan_tzan_enum,theo_lapl}  and structural theorems \cite{terao_exps,abe_euler,desc_alg,stump_free} associated to a finite Coxeter group $W$ can be phrased in terms of the reflection arrangement $\AAA_W$. Moreover, numerical invariants associated to reflection arrangements seem to behave particularly well. For example, the characteristic polynomial of $\AAA_W$ will always factor as $ \chi(\AAA_W,t)=\prod_{i=1}^n(t-e_i)$, where $n$ is the rank of the group $W$ and the \defn{$e_i$}'s are certain positive integers, known as the \defn{exponents} of $W$. We write them in increasing order $e_i\leq e_{i+1}$ and we always have $e_1=1$ because the arrangement $\AAA_W$ is central. From this information the definition in \eqref{eq:intro:defn:beta} implies that the beta invariant of $\AAA_W$ is given by
\begin{equation}
    \beta(\AAA_W)=\prod_{i=2}^n(e_i-1).\label{eq:intro:b(A_W)}
\end{equation}

The main motivation behind this paper is the realization that the right hand side of \eqref{eq:intro:b(A_W)} is a factor in the beautiful product formula below, due to Chapoton \cite{chapoton}, for the number of reflections of full support in a finite Coxeter group $W$. Recall that an element $g$ in a Coxeter group $W$ with simple generators $S$ is called \defn{full support} if all (equivalently, any of) its reduced $S$-decompositions involve the full set $S$ of simple reflections of $W$. 

\begin{proposition}[{\cite[Prop.~2.1]{chapoton}}]\label{Prop:intro:chap_form}
The number $f_W$ of reflections of full support in a finite, irreducible Coxeter group $W$ of rank $n$ is given by the formula
\begin{equation}
f_W=\dfrac{nh}{|W|}\cdot\prod_{i=2}^n(e_i-1),\label{eq:intro:chapoton_form}
\end{equation}
where $e_1,\ldots,e_n$ are the exponents of $W$ and $h:=e_n+1$ is the Coxeter number of $W$.
\end{proposition}

Chapoton proved this formula in a case-by-case fashion relying on the classification of finite Coxeter groups and for many years no case-free proof was known. Thiel \cite{thiel} gave a uniform proof for Weyl groups $W$, but only after relating the reflections of full support with certain collections of order ideals in the root poset of $W$ that had been uniformly enumerated by Sommers \cite{sommers}. As a special case of our main Theorem~\ref{thm:main}, we will give a different proof of Chapoton's formula via a double-counting argument that exploits a combinatorial interpretation of the beta invariant  $\beta(\AAA_W)$. This proof is self-contained, case-free, and applies to all finite Coxeter groups (not just Weyl groups). Moreover, it is brief enough that we can sketch it below in \S\ref{sec:intro:proof} in just a few paragraphs. In \S\ref{sec:intro:main_thm} we present our main theorem, a generalization of Chapoton's formula where we count certain \emph{parabolic subgroups} $G\leq W$ that belong to given conjugacy classes and are of full support (with Proposition~\ref{Prop:intro:chap_form} corresponding to the special case that $G$ consists of the identity element and a single reflection).

\subsection{A case-free proof of Chapoton's formula}
\label{sec:intro:proof}

The first important ingredient of the proof is a new interpretation  (see Lemma~\ref{lem:geom_interpr_support} and Corollary~\ref{cor:suppor_interp}) of the \emph{full support} property in finite Coxeter groups $W$ with simple system $S$. Let $\AAA_W$ be the associated reflection arrangement of $W$ with ambient space $V$ and collection of chambers $\CCC(\AAA_W)$. If $C_0\in\CCC(\AAA_W)$ is the \emph{fundamental chamber} of $\AAA_W$ (determined by $S$), then a reflection $t\in W$ that is full support cannot fix any \emph{non-trivial} face $F$ of $C_0$; this would imply that $t$ has a reduced $S$-decomposition involving only the subset $J\subseteq S$ of simple generators that fix $F$. This argument also applies in the opposite direction, so that a reflection $t \in W$ is full support \emph{if and only if} its fixed hyperplane $H:=V^t$ intersects the fundamental chamber $C_0$ precisely at the origin $\bm 0\in V$.

It turns out that this condition is very much related to an enumerative  interpretation of the beta invariant. Slightly recasting a theorem of Greene and Zaslavksy \cite{greene_zaslavsky}, we show in Corollary~\ref{cor:beta=half_chambers} that for a real, essential, central arrangement $\AAA$ the number $\beta(\AAA)$ equals \emph{half} the quantity of chambers of $\AAA$ that intersect any fixed hyperplane $H\in\AAA$ precisely at the origin $\bm 0\in V$.

Now, if $\defn{N}:=|\AAA_W|$ denotes the number of reflections in $W$, equivalently hyperplanes in $\AAA_W$, the well known relation $hn=2N$ allows us --along with \eqref{eq:intro:b(A_W)}-- to rewrite Chapoton's formula \eqref{eq:intro:chapoton_form} as
\begin{equation}
    |W|\cdot f_W=N\cdot2\beta(\AAA_W).\label{eq:intro:|W|f_W=blah}
\end{equation}
The two sides of this equation reflect two different ways to enumerate the elements of the set 
\[
\mathfrak{P}(W):=\Big\{ (H,C)\in\AAA_W\times\CCC(\AAA_W)\ \ \text{such that}\ \ H\cap C=\bm 0\Big\}.
\]
Counting pairs $(H,C)\in\mathfrak{P}(W)$ starting with the hyperplane $H$ and then all suitable chambers $C$ we find a total of $N\cdot 2\beta(\AAA_W)$ elements, after the interpretation of $\beta(\AAA_W)$ we discussed. On the other hand, by our geometric interpretation of the full support property, the number of hyperplanes $H$ intersecting the fundamental chamber $C_0$ at the origin equals $f_W$. Moreover, the intersection patterns of hyperplanes and chambers are the same for all chambers and so if we count pairs $(H,C)\in \mathfrak{P}(W)$ starting with the chamber $C$ and then considering all suitable hyperplanes $H$, we will find a total of $|W|\cdot f_W$ elements.
\subsection{Our Main Theorem}
\label{sec:intro:main_thm}

After the previous discussion, it is natural to ask whether we can further apply this geometric interpretation of support and count other objects in finite Coxeter groups. Moreover, it is well known that for any flat $X\in\AAA_W$, the restricted arrangements $\AAA_W^X$ also have characteristic polynomials that factor in positive integer roots $b_1^X,\ldots,b_k^X$, $k=\dim(X)$, which are known as the Orlik-Solomon exponents of $\AAA^X$ (see \S\ref{ssec:OS-exps}). Is there a way to generalize Chapoton's formula \eqref{eq:intro:chapoton_form} using these exponents $b_i^X$?

It turns out that both these questions have positive answers. We briefly preview some definitions and notation before stating our main theorem (see \S\ref{sec:Coxeter_appl} for more details). In a finite Coxeter group $W$, we consider for any parabolic subgroup $W_Y\leq W$ two subsets $I,J\subseteq S$ of the set $S$ of simple reflections, called the \defn{core} and \defn{support} of $W_Y$ respectively as follows. We require that
\[
\langle I\rangle \leq W_Y\leq \langle J\rangle,
\]
and that $I$ is the maximum subset with this property and $J$ the minimum one.

The geometric interpretation of support discussed in \S\ref{sec:intro:proof} holds in this case too, so that $J=S$ for a flat $Y$ if and only if $Y\cap C_0=\bm 0$. We then call such flats \defn{faraway flats} for the fundamental chamber $C_0$. Now, it is easy to see that the core $I$ of $Y$ can have size at most equal to $\op{rank}(W_Y)-1$ if $W_Y$ is full support. We call such flats \defn{nearest faraway flats} for the fundamental chamber.

Our main theorem gives the enumeration of nearest faraway flats for the fundamental chamber, keeping track of the parabolic types of the flat and of its core.

\begin{restatable*}[]{theorem}{mainthm}
\label{thm:main}
Let $W$ be an irreducible, finite Coxeter group with reflection arrangement $\AAA$ and let $X$ and $Y$ be two flats in $\LLL_{\AAA}$ such that $\dim(X)=\dim(Y)+1$. Then, the set $\nfw\big([X]\big)_{[Y]}$, which consists of the parabolic subgroups of $W$ of type $[Y]$ that are full support and have core of parabolic type $[X]$ (see also Defn.~\ref{defn:sets:g(I)y}) has size given by the formula
\[
\Big|\nfw\big([X]\big)_{[Y]}\Big|=\dfrac{2\cdot u_{[X],[Y]}}{[N(X):W_X]}\cdot\prod_{i=2}^{\dim(X)}(b_i^X-1),
\]
where $N(X)$ and $W_X$ are respectively the setwise and pointwise stabilizers of $X$, $b_i^X$ denote the Orlik-Solomon exponents of the restricted arrangement $\AAA^X$, and the Orlik-Solomon number $u_{[X],[Y]}$ is defined to be the number of hyperplanes in $\AAA^X$ of parabolic type $[Y]$.
\end{restatable*}

Notice that the formula of Theorem~\ref{thm:main} gives a full generalization of Chapoton's formula \eqref{eq:intro:chapoton_form} (in fact of Thiel's generalization of it in \cite[Thm.~1.2]{thiel}) for any parabolic type. The fact alone that the quantity on the right hand side of the equation in Theorem~\ref{thm:main} is an integer is surprising (and as far as we know unknown till now).

\subsection*{Summary} In \S\ref{sec:double_counting} we present our interpretation of the beta invariant and the notions of faraway and nearest faraway flats in the general setting of real, central, simplicial arrangements. As we briefly discuss in \S\ref{ssec:free_arrts} we hope and expect to see applications of these ideas outside the context of Coxeter groups. In \S\ref{sec:Coxeter_appl} we translate the results of \S\ref{sec:double_counting} to the setting of our main Theorem~\ref{thm:main} and prove it. We finish in \S\ref{ssec:extras} with certain direct generalizations of Theorem~\ref{thm:main} and the discussion of possible further extensions of it.

\section{Double counting pairs of chambers and faraway planes}
\label{sec:double_counting}

In this section, we introduce the basic combinatorial objects of this paper and prove our main technical lemma (Lemma~\ref{lem: double_counting}). For any terminology or statements that are not explained here, the reader may consult the standard references \cite{OT_book,stanley_book,aguiar_book}.

\subsection{Hyperplane arrangements, faraway planes, and nearest faraway flats}
\label{ssec:definitions_far_nfar}

We will start first by fixing some notation and terminology. In this section, $\AAA$ will denote a real, essential, central hyperplane arrangement and $\CCC(\AAA)$ its collection of chambers. We will write $V$ for its ambient real vector space and $\bm 0$ for the origin of $V$. The intersection lattice of the arrangement $\AAA$ will be denoted by $\LLL_{\AAA}$ and we will call its elements $X\in\LLL_{\AAA}$ the flats of the arrangement. For any such flat $X\in\LLL_{\AAA}$, the \emph{restricted arrangement} $\AAA^X$ is a hyperplane arrangement with ambient space $X$, defined as the collection of top dimensional flats properly contained in $X$. More formally, we have
\[
\AAA^X:=\{H\cap X\ |\ H\in\AAA\ ,\ H\not\supseteq X\}.
\]
For any face $F\subseteq C$ of a chamber $C\in\CCC(\AAA)$ of the arrangement, $\spn(F)$  will denote the linear span of $F$, which is a flat in the intersection lattice $\LLL_{\AAA}$. Notice that $F$ itself is then a chamber of the restricted arrangement $\AAA^{\spn(F)}$. We will denote by $\dim(X)$ and $\dim(F)$ the respective dimensions of a flat $X$ and face $F$ (notice that $\dim(F)=\dim\big(\spn(F)\big)$).

We are now ready to give the proper definitions for the notions of faraway planes and nearest faraway flats discussed in the introduction.

\begin{definition}\label{defn: far planes}
For a real, essential, central hyperplane arrangement $\AAA$, and a chamber $C\in\CCC(\AAA)$, we say that a hyperplane $H\in\AAA$ is a \defn{faraway plane} for the chamber $C$ if it intersects $C$ only at the origin; that is, if $H\cap C=\bm 0$. We will write $\defn{\far_{\AAA}(C)}$ (or simply $\far(C)$ when there can be no confusion) for the collection of faraway planes associated to the chamber $C$. 

We will often want to restrict --without altering the ambient arrangement $\AAA$-- the collection of hyperplanes in which we count how many are faraway planes. For any subset $\PPP\subseteq \AAA$ we define \defn{$\far_{\AAA,\PPP}(C)$} to be the set of faraway planes for the chamber $C$ \emph{that belong to $\PPP$}.
\end{definition}

Similarly for an arrangement $\AAA$ as above and a chamber $C\in\CCC(\AAA)$, we will call a flat $X\in\LLL_{\AAA}$ a \defn{faraway flat} for $C$ if $X\cap C=\bm 0$. We are interested in a particular subset of faraway flats whose definition (Defn.~\ref{defn: nfar flats}) we motivate now. Notice that for any flat $X\in\LLL_{\AAA}$ and any chamber $C\in\CCC(\AAA)$, there must exist faces of $C$ whose span contains $X$ and among those at least one \emph{minimum-dimensional} such face $F$. Of course, we must have $\dim(F)\geq \dim(X)$ and equality would force that $C\cap X=F$ so that in the case of a flat $X$ that is \emph{faraway for $C$}, we must have have $\dim(F)\gneq \dim(X)$. It is natural to consider the faraway flats $X$ for which $\dim(F)=\dim(X)+1$ as being \emph{closest} or \emph{nearest} to the chamber $C$.

\begin{definition}\label{defn: nfar flats}
Let $\AAA$ be a real, essential, central hyperplane arrangement, with collection of chambers $\CCC(\AAA)$ and intersection lattice $\LLL_{\AAA}$. For a given chamber $C\in \CCC(\AAA)$ we will say that a flat $X\in\LLL_{\AAA}$ is a \defn{nearest faraway flat} for $C$ if it is a faraway flat (i.e., if $X\cap C=\bm 0$) and the minimum dimension of any face $F$ of $C$ whose span $\spn(F)$ contains $X$ is equal to $\dim(X)+1$.
\end{definition}

\begin{remark}\label{rem:0_is_not_nearest_far_flat}
Comparing Definitions~\ref{defn: far planes} and ~\ref{defn: nfar flats}, we see that the origin $\bm 0$ is considered a faraway flat (for any chamber) but it is not a \emph{nearest} faraway flat.
\end{remark}

Now, in general there may be multiple faces $F\subseteq C$ that satisfy the criteria of Defn.~\ref{defn: nfar flats} for a given flat $X$. For instance, take the arrangement of hyperplanes in $\mathbb{R}^3$ given as 
\[
\AAA_{sq}:=\{x-z=0,x+z=0,y-z=0,y+z=0\}
\]
(these four hyperplanes are the linear spans of the four sides of a \emph{square} pyramid with apex at the origin). If $C\in\CCC(\AAA_{sq})$ is the chamber that contains the positive $z$ axis in its interior, then the $1$-dimensional flat $X:=\{x=z=0\}$ (i.e., the $y$-axis) is a nearest faraway flat for $C$ and lives in the span of two distinct $2$-dimensional faces of $C$ (those that correspond to the hyperplanes $x-z=0$ and $x+z=0$). 

There is however, a class of arrangements for which each nearest faraway flat has a uniquely determined associated face $F$ as in Defn.~\ref{defn: nfar flats}. Recall that an essential, central hyperplane arrangement is called \defn{simplicial} when all its chambers are simplicial cones; i.e., cones whose rays form a basis of the ambient space $V$.

\begin{lemma}
Let $\AAA$ be a central, essential, simplicial, hyperplane arrangement, $C\in\CCC(\AAA)$ one of its chambers, and $X\in\LLL_{\AAA}$ one of its flats. If $X$ is a nearest faraway flat for $C$, then there exists a unique face $F$ of $C$ that satisfies $\spn(F)\supseteq X$ and $\dim(F)=\dim(X)+1$.
\end{lemma}
\begin{proof}
Assume on the contrary that there are two distinct faces $F_1$ and $F_2$ of $C$ satisfying the conditions of the statement. Clearly, we must have $\spn(F_1)\neq \spn(F_2)$ so that the assumed conditions force  $\spn(F_1)\cap\spn(F_2)=X$. Now, since $\AAA$ is simplicial, we must further have that
\[
\spn(F_1)\cap\spn(F_2)=\spn(F_1\cap F_2),
\]
because the union of the rays that span the faces $F_1$ and $F_2$ is a linearly independent collection. But the intersection $F_1\cap F_2$ is itself a face of $C$ and the equality $X=\spn(F_1\cap F_2)$ contradicts the assumption that $X$ is a faraway flat. This completes the proof.
\end{proof}

This allows us to give the following definitions.

\begin{definition}\label{defn: associated_face}
Let $\AAA$ be a real, essential, central, \emph{simplicial} hyperplane arrangement. For a chamber $C\in\CCC(\AAA)$ and a flat $X\in\LLL_{\AAA}$ that is a nearest faraway flat for $C$, we call the \emph{unique} face $F$ that satisfies the conditions of Defn.~\ref{defn: nfar flats} the \defn{associated face} to the (nearest faraway) flat $X$. We write \defn{$\nfar_{\AAA}(C,F)$} (or $\nfar(C,F)$ when there can be no confusion) for the collection of nearest faraway flats for the chamber $C$, with associated face $F$. 

As in Defn.~\ref{defn: far planes} we will often want to restrict the family of flats that are considered. For any subset $\QQQ\subseteq \LLL_{\AAA}$, we define \defn{$\nfar_{\AAA,\QQQ}(C,F)$} to be the set of nearest faraway flats for the chamber $C$ and with associated face $F$, \emph{that belong to $\QQQ$}.
\end{definition}

\begin{remark}\label{rem:faraway planes are always nearest faraway flats}
Comparing Definitions~\ref{defn: far planes},~\ref{defn: nfar flats}, and~\ref{defn: associated_face}, notice that --given a simplicial\footnote{It is not necessary to assume simpliciality here since there can only be a single full-dimensional face in any chamber; we do it to keep the statement formally compatible with Defn.~\ref{defn: associated_face}.} arrangement $\AAA$ and a chamber $C$ as in the discussion so far-- a faraway plane is always a nearest faraway flat with associated face the whole chamber $C$. In particular, we will have for any subset $\PPP\subseteq \AAA$ that
\[
\far_{\AAA,\PPP}(C)=\nfar_{\AAA,\PPP}(C,C),
\]
where we treat $\PPP\subseteq \AAA$ also as a collection of flats, i.e. as a subset of $\LLL_{\AAA}$.
\end{remark}

On the opposite direction to Remark~\ref{rem:faraway planes are always nearest faraway flats} above, the condition that faraway flats be \emph{nearest} allows us to relate them to faraway planes in restricted arrangements. In particular, the number of nearest faraway flats for a chamber $C$, with associated face $F$, must equal the number of faraway planes in the restricted arrangement $\AAA^{\spn(F)}$ for the chamber $F\in\CCC(\AAA^{\spn(F)})$. We formalize this below.

\begin{proposition}\label{prop:nfar-to-far}
Let $\AAA$ be a simplicial hyperplane arrangement as in Defn.~\ref{defn: associated_face} with a chosen chamber $C\in\CCC(\AAA)$ and face $F\subseteq C$. If $\QQQ\subseteq \LLL_{\AAA}$ is any subset of flats with $\bm 0\notin \QQQ$, we will have
\[
\far_{\AAA^{\spn(F)},\QQQ\cap\AAA^{\spn(F)}}(F)=\nfar_{\AAA,\QQQ}(C,F).
\]
\end{proposition}
\begin{proof}
The requirement $\bm 0\notin \QQQ$ is to avoid the phenomenon of Remark~\ref{rem:0_is_not_nearest_far_flat}. 

To prove the statement, we will show that the corresponding collections of faraway planes in $\AAA^{\spn(F)}$ and nearest faraway flats in $\AAA$ are equal by demonstrating the inclusions $\subseteq$ and $\supseteq$ between them. To show $\supseteq$, we start with a flat $X\in\nfar_{\AAA,\QQQ}(C,F)$. This means that $X\cap C=\bm 0$, which further implies $X\cap F=0$ since $F$ is a face of $C$. Moreover by Definition~\ref{defn: associated_face} $X$ also satisfies $\dim(X)=\dim(F)-1$ and $X\subseteq \spn(F)$, which completes the proof for this direction. 

To prove the inclusion $\subseteq$, we start with a hyperplane $Z\in\QQQ\cap\AAA^{\spn(F)}$ that satisfies $Z\cap F=\bm 0$. Since $Z\subseteq \spn(F)$ we will have that
\[
Z\cap C\subseteq \spn(F)\cap C=F,
\]
where the last equality holds because $F$ is a face of the (convex) chamber $C$. Now $Z\cap C\subseteq F$ implies that $ Z\cap C=Z\cap F=\bm0 $, which completes the direction $\subseteq$ and thus the proof.
\end{proof}

\subsection{Crapo's beta invariant and the double counting lemma}\label{ssec:double_counting_lemma}

We start as in \S\ref{ssec:definitions_far_nfar} with a real, essential (but not necessarily central or simplicial) hyperplane arrangement $\AAA$ with intersection lattice $\LLL_{\AAA}$ and ambient space $V$. Recall that its \defn{characteristic polynomial} $\chi(\AAA,t)$ is defined in terms of the M{\"o}bius function $\mu$ on $\LLL_{\AAA}$ (where the order is inverse inclusion of flats) as
\begin{equation}
    \defn{\chi(\AAA,t)}:=\sum_{X\in\LLL_{\AAA}}\mu(V,X)\cdot t^{\dim(X)}.
\end{equation}
As is well known, many important combinatorial invariants of the arrangement $\AAA$ are encoded in its characteristic polynomial. For instance the number $\defn{b(\AAA)}$ of \emph{bounded} chambers of $\AAA$ is given as 
\begin{equation}
    b(\AAA)= (-1)^{\op{rank}(\AAA)}\cdot \chi(\AAA,1),
\end{equation}
where the \defn{rank} of the arrangement is the dimension $\dim(V)$ of the ambient space\footnote{Nore generally, the rank of an arrangement $\AAA$ equals the difference between the dimension of the ambient space and the center of $\AAA$; in our case, since we assume $\AAA$ to be essential, its center --the origin $\bm 0$-- is $0$-dimensional.}. In this paper we are interested in a similar invariant, defined in terms of the \emph{derivative} of the characteristic polynomial. Now, we further assume that $\AAA$ is central.

\begin{definition}[Crapo's beta invariant]\label{defn:beta(A)}
For a real, essential, central hyperplane arrangement $\AAA$ with characteristic polynomial $\chi(\AAA,t)$, we define its \defn{beta invariant} $\beta(\AAA)$ as
\[
\defn{\beta(\AAA)}:=(-1)^{\op{rank}(\AAA)-1}\cdot \dfrac{\op{d}\chi(\AAA,t)}{\op{d}t}\Big|_{t=1},
\]
where $\op{d}/\op{dt}$ denotes differentiation with respect to $t$. 
\end{definition}

It is not very difficult to see that $\beta(\AAA)$ must be a \emph{positive} integer; much more than that, it has an enumerative interpretation. The following proposition is \cite[Thm.~3.4]{greene_zaslavsky} but essentially its proof comes down to \cite[Theorem~D]{zaslavsky_facing_up}. It also appears as Exercise~22:(d) in \cite[Lec.~4]{stanley_book}.

\begin{proposition}[{\cite[Thm.~3.4]{greene_zaslavsky}}]\label{Prop: beta interpretation}
Let $\AAA$ be a real, essential, central hyperplane arrangement and let $H'$ be a proper translate of an arbitrary hyperplane $H\in\AAA$. Then, Crapo's beta invariant for $\AAA$ agrees with the number of bounded chambers in the arrangement $\AAA\cup \{H'\}$; that is, we have
\[
\beta(\AAA)=b\left(\AAA\cup\{H'\}\right).
\]
\end{proposition}

\begin{corollary}\label{cor:beta=half_chambers}
Let $\AAA$ be as in Proposition~\ref{Prop: beta interpretation} and let $H\in\AAA$ be any given hyperplane. Then, the beta invariant $\beta(\AAA)$ equals half the number of chambers whose intersection with $H$ is the origin $\bm 0$.
\end{corollary}
\begin{proof}
With $V$ being the ambient space of the arrangement, we start by choosing a linear form $\alpha_H\in V^*$ such that $\alpha_H(v)=0$ for all points $v\in H$. Because $\AAA$ is central, the hyperplane $H$ divides its collection of chambers into two disjoint sets $\CCC_{H^+}(\AAA)$ and $\CCC_{H^-}(\AAA)$ that consist of the chambers $C$ whose points $v\in C$ satisfy $\alpha_H(v)\geq 0$ or $\alpha_H(v)\leq 0$ respectively.

Now, if we define $H'$ to be the translate of $H$ cut by the equation $\alpha_H(v)=1$, the chambers of the arrangement $\AAA\cup H'$ are divided in two sets as follows. First, we have the chambers in $\CCC_{H^-}(\AAA)$ as they are, but then, for each chamber $K\in\CCC_{H^+}(\AAA)$ there exist two chambers $K_1,K_2$ in $\CCC(\AAA\cup H')$; the points $v$ in $K_1$ satisfy $0\leq \alpha_H(v)\leq 1$, while the points $v\in K_2$ give $1\leq \alpha_H(v)$. Now, the chamber $K_2$ is always unbounded, but the chamber $K_1$ is bounded if and only if it has a finite intersection with the hyperplane $H'$. If the intersection $K_1\cap H'$ is finite, then since $\AAA$ is central and $H'$ is a (parallel) translate of the hyperplane $H\ni\bm 0$, we must have $H\cap K=\bm 0$ (for instance by Thales' theorem on similar triangles). If the intersection is infinite it means that $K_1$ and thus $K$ contain an infinite ray parallel to $H$, and since $K$ is a cone this means it must contain an infinite ray \emph{inside} $H$ as well; that is $K\cap H\neq \bm 0$.

This means that the set of bounded chambers in $\AAA\cup H'$ is in bijection with the set of chambers $K\in \CCC_{H+}(\AAA)$ such that $K\cap H=\bm 0$. Because $\AAA$ is central, it is invariant under scalar multiplication by $-1$, and then the same statement must be true for the arrangement $\AAA\cup H''$, with $H''$ cut by $\alpha_H(v)=-1$, and the set $\CCC_{H^-}(\AAA)$; moreover, the two cardinalities $b(\AAA\cup H')$ and $b(\AAA\cup H'')$ must be equal. Combining this with Proposition~\ref{Prop: beta interpretation}
 we have that the quantity $2\cdot \beta(\AAA)$ is equal to the total number of chambers $C\in\CCC(\AAA)$ such that $C\cap H=\bm 0$. This completes the proof.
\end{proof}

\begin{remark}
The interpretation of the beta invariant given in Corollary~\ref{cor:beta=half_chambers} above is used in \cite[\S8.4]{aguiar_book} as the definition of $\beta(\AAA)$. In this setting they relate it to the determinant of the Varchenko matrix of $\AAA$ \cite[Thm.~8.11]{aguiar_book}. The connection to Crapo's definition is briefly discussed at the end of Section~8 in \cite{aguiar_book}.
\end{remark}

Below, we give the key technical lemma of this paper that we will rely upon in Section~\ref{sec:Coxeter_appl} for the proof of our main Theorem~\ref{thm:main}.

\begin{lemma}[Double counting lemma]\label{lem: double_counting}\ \newline
Let $\AAA$ be a real, central hyperplane arrangement with set of chambers $\mathcal{C}(\AAA)$, and let $\PPP\subseteq\AAA$ be an arbitrary subset of hyperplanes. If $\far_{\AAA,\PPP}(C)$ denotes the set of hyperplanes in $\PPP$ that are \emph{faraway planes} for some chamber $C\in\mathcal{C}(\AAA)$, we will have that
\[
\sum_{C\in\mathcal{C}(\AAA)}\Big|\far_{\AAA,\PPP}(C)\Big|=2\cdot|\PPP|\cdot\beta(\AAA),
\]
where $|\PPP|$ and $\beta(\AAA)$ denote the size of $\PPP$ and the beta invariant of $\AAA$ respectively.
\end{lemma}
\begin{proof}
The proof is the same as our discussion in \S\ref{sec:intro:proof}. The two sides of the equation reflect two different ways to enumerate the elements of the set 
\[
\mathfrak{P}(\PPP,\AAA):=\Big\{ (H,C)\in\PPP\times\CCC(\AAA)\ \ \text{such that}\ \ H\cap C=\bm 0\Big\}.
\]
Counting pairs $(H,C)\in\mathfrak{P}(\PPP,\AAA)$ starting with the hyperplane $H$ and then finding all chambers $C$ such that $H\cap C=\bm 0$, we get a total of $2\cdot|\PPP|\cdot\beta(\AAA)$ elements, after Corollary~\ref{cor:beta=half_chambers}. On the other hand, enumerating the pairs $(H,C)\in\mathfrak{P}(\PPP,\AAA)$ starting with the chamber $C$ and counting the hyperplanes $H\in\PPP$ for which $H\cap C=\bm 0$ gives us, after Definition~\ref{defn: far planes}, precisely $\sum_{C\in\CCC(\AAA)}\big|\far_{\AAA,\PPP}(C)\big|$. This completes the proof.
\end{proof}

\subsection{General numerological applications}\label{ssec:free_arrts}
The material presented in this section will be mostly applied --in this paper-- in the context of reflection arrangements and their restrictions. We chose to discuss it in higher generality (arbitrary central, or simplicial central arrangements) because we hope and expect that it will have further applications. We hint at a few possibilities here.

There is a large family of central hyperplane arrangements $\AAA$, called \defn{free arrangements} \cite{yoshi_free}, for which the characteristic polynomial $\chi(\AAA,t)$ factors into linear terms with positive integer roots. They are defined by a condition on the ring $\mathcal{D}(\AAA)$ of (polynomial) vector fields tangent to the hyperplanes of $\AAA$; if the ring is a \emph{free} module over the ambient algebra $\mathbb{C}[V]$, then we say that $\AAA$ is free. Terao \cite{terao_exps} proved the factorization property discussed above by relating the roots of $\chi(\AAA,t)$ to the homogeneous degrees of a basis for the module $\mathcal{D}(\AAA)$. For a free arrangement $\AAA$ of rank $n$, the roots $\varepsilon_1,\ldots,\varepsilon_n$ of $\chi(\AAA,t)$ are called the \defn{exponents} of $\AAA$. Because $\AAA$ is central, we may assume that $\varepsilon_1=1$ and then immediately by Defn.~\ref{defn:beta(A)} we will have the following statement.

\begin{proposition}\label{prop:beta(A)_A_free}
For a free arrangement $\AAA$ of rank $n$, with exponents $\varepsilon_1=1,\varepsilon_2,\ldots,\varepsilon_n$, the beta invariant $\beta(\AAA)$ is given by 
\[
\beta(\AAA)=\prod_{i=2}^n(\varepsilon_i-1).
\]
\end{proposition}

Because of Proposition~\ref{prop:beta(A)_A_free} many enumerative questions regarding the sets defined in \S\ref{ssec:definitions_far_nfar} will have answers given in terms of product formulas when $\AAA$ is a free arrangement. We give the following as an example.

\begin{corollary}\label{cor:expected_far_planes:free_arr}
For a free hyperplane arrangement $\AAA$ of rank $n$ with exponents $\varepsilon_1=1,\varepsilon_2,\ldots,\varepsilon_n$, the \emph{average} number of faraway planes for a chamber is given by the formula
\[
\op{Exp}_{C\in\CCC(\AAA)}\left(\Big|\far_{\AAA}(C)\Big|\right)=\big|\AAA\big|\cdot \prod_{i=2}^n\dfrac{\varepsilon_i-1}{\varepsilon_i+1}.
\]
\end{corollary}
\begin{proof}
This is an immediate corollary of Zaslavsky's theorem that $\big|\CCC(\AAA)\big|=\prod_{i=1}^n(\varepsilon_i+1)$ and Lemma~\ref{lem: double_counting} for $\PPP=\AAA$. Notice that the factor $2$ of the lemma canceled with the factor $e_1+1=2$.
\end{proof}

\begin{remark}
Many popular classes of arrangements are free; for instance all chordal graphical arrangements --more generally all supersolvable arrangements \cite[Thm.~4.58]{OT_book}-- all reflection arrangements and all restricted reflection arrangements (see the discussion below \eqref{eq:defn:OS_exps}) are free.
\end{remark}

\section{Applications in finite Coxeter groups}
\label{sec:Coxeter_appl}

In this section we prove our main result (Thm.~\ref{thm:main}) by combining  the double counting lemma of the previous section (Lemma~\ref{lem: double_counting}) with a new geometric interpretation of the notion of \emph{support} in finite Coxeter groups (see Lemma~\ref{lem:geom_interpr_support}). In what follows, the presentation assumes some familiarity with the theory of finite Coxeter groups; for any terminology or statements that are not explained, the reader may consult the standard references \cite{Humphreys,OT_book,Kane}.

\subsection{Finite Coxeter groups and their parabolic subgroups}\label{ssec:coxeter_stuff}
We start again by fixing some notation and terminology. We will write $W$ to denote an irreducible, finite Coxeter group; $S$ for its set of simple generators, and $n:=|S|$ for its rank. For any such $W$, there exists a natural real vector space $V\cong\mathbb{R}^n$ on which $W$ is acting via linear orthogonal transformations. In this way, the collection of finite Coxeter groups $W$ agrees precisely with the collection of finite real reflection groups (namely, the finite subgroups of $\op{GL}(V)$ generated by Euclidean reflections). We will denote by \defn{$\AAA_W$} the \defn{reflection arrangement} of $W$; it is a real, essential, central, simplicial hyperplane arrangement that consists of the fixed hyperplanes of the reflections of $W$ and whose ambient space is $V$. 

We will denote by \defn{$C_0$} the \defn{fundamental chamber} of $W$; namely, the chamber of $\AAA_W$ whose boundary hyperplanes are exactly the fixed hyperplanes of the simple generators $s\in S$.
Any subset $J\subset S$ determines a face \defn{$C_0^J$} of the fundamental chamber as the intersection $C_0\cap\bigcap_{s\in J}V^s$. The subgroups $\langle J\rangle\leq W$ are precisely the pointwise stabilizers of the faces $C_0^J$ and are called \defn{standard parabolic subgroups}. More generally the pointwise stabilizer $W_B$ of \emph{any} collection of points $B\subseteq V$ will be called a \defn{parabolic subgroup}. It turns out that the fixed spaces of parabolic subgroups are always flats of the reflection arrangement $\AAA_W$ (with each flat $X\in\LLL_{\AAA_W}$ determining a distinct subgroup $W_X$) and that the parabolic subgroups $W_X$ are themselves reflection groups, with $\op{rank}(W_X)=\op{codim}(X)$, which are always conjugate to some standard parabolic subgroup (and it is possible that different standard parabolic subgroups are conjugate to each other).

The group $W$ acts via conjugation on the collection of parabolic subgroups $W_X\leq W$ and this action is equivariant to its natural action by multiplication on the flats $X\in\LLL_{\AAA_W}$. The orbits of these actions are called \defn{parabolic types} (generalizing the \emph{cycle type} of elements or Young subgroups in the symmetric group $S_n$) and will be denoted by $[X]\in\LLL_{\AAA_W}/W$ (for any representative flat $X$). The isomorphism type of a parabolic subgroup $W_X\leq W$ as a finite Coxeter group usually determines its parabolic type as well, but not always. For example, in the hyperoctahedral groups $B_n, n\geq 2$, any parabolic subgroup $W_H$ that fixes a hyperplane $H\in\LLL_{\AAA_{B_n}}$ is isomorphic to $A_1\cong S_2$ as a Coxeter group, but there are always two distinct parabolic types $[H_1]$ and $[H_2]$ that it could belong to (corresponding to the two orbits of hyperplanes, associated to the long and short roots respectively).

The collection $\left( \langle J\rangle \right)_{J\subseteq S}$ of standard parabolic subgroups of a finite Coxeter group $W$ forms a boolean lattice under inclusion. Therefore, for any parabolic subgroup $W_X\leq W$ (in fact any subgroup) there exist unique subsets $I$ and $J$ of $S$ for which
\begin{equation}
\langle I\rangle \leq W_X\leq \langle J\rangle,\label{eq:defn:I_and_J}
\end{equation}
and such that $I$ is the maximum such subset and $J$ the minimum one. That is, $\langle I\rangle$ is the largest standard parabolic subgroup \emph{contained} in $W_X$ and $\langle J\rangle$ is the smallest standard parabolic subgroup \emph{containing} $W_X$. The subset $J$ as above is known \cite[\S2.1.1]{verges_biane} as the \defn{support} of $W_X$ and if $J=S$ we say that $W_X$ is \defn{full support}. To facilitate discussing these objects, we will call $I$ the \defn{core} of $W_X$.

\begin{definition}\label{defn:sets:g(I)y}
In a finite Coxeter group $W$ with set of simple generators $S$ and for a subset $I\subseteq S$, we define $\defn{\nfw(I)}$ to be the collection of parabolic subgroups that are full support and whose core is $I$. More generally, if $[X]\in\LLL_{\AAA_W}/W$ is an orbit of flats, we write $\defn{\nfw\big([X]\big)}$ for the collection of parabolic subgroups of full support whose core $I$ determines a subgroup $\langle I \rangle$ of parabolic type $[X]$.

\noindent Furthermore, if $[Y]\in\LLL_{\AAA_W}/W$ is another orbit of flats, we will write $\defn{\nfw(I)_{[Y]}}$ and $\defn{\nfw\big([X]\big)_{[Y]}}$ for the subsets of $\nfw(I)$ and $\nfw\big([X]\big)$ respectively, consisting of those subgroups that have parabolic type $[Y]$. 
\end{definition}

A natural problem is to compute the cardinalities of the sets in Defn.~\ref{defn:sets:g(I)y} above. It turns out that for a particular subfamily, this is achievable by using the techniques of Section~\ref{sec:double_counting}. In our main Theorem~\ref{thm:main} we will compute the cardinalities of the sets $\nfw\big([X]\big)_{[Y]}$ for pairs of parabolic types $\big([X],[Y]\big)$ that satisfy $\dim(X)=\dim(Y)+1$ (for any representatives $X\in [X]$ and $Y\in [Y]$). 

There is a more conceptual way to describe the family discussed in the previous paragraph. Following Taylor \cite[Defn.~1.1]{taylor} we will say for two reflection subgroups $H\leq K\leq W$ of $W$, that $K$ is a \defn{simple extension} of $H$ if $K=\langle H, r\rangle$ for some (\emph{not necessarily simple}) reflection $r$ of $W$ such that $r\not\in H$. Then, given two parabolic subgroups $W_X\leq W_Y$ of $W$, we have that $\dim(X)=\dim(Y)+1$ if and only if $W_Y$ is a simple extension of $W_X$ (this can be seen for instance by simultaneously conjugating $W_X$ and $W_Y$ to standard parabolic subgroups, as in \cite[Prop.~2.4]{verges}).

This motivates the following definition and distinct notation, for those parabolic subgroups of full support that are simple extensions of their cores.

\begin{definition}\label{defn:nfwse}
In a finite Coxeter group $W$ with set of simple generators $S$ and for a subset $I\subseteq S$, we define $\defn{\nfwse(I)}$ to be the collection of parabolic subgroups of $W$ that have core $I$, are of full support, and which are \emph{simple extensions} of $\langle I\rangle$. More generally, if $[X]\in\LLL_{\AAA_W}/W$ is an orbit of flats, we write $\defn{\nfwse\big([X]\big)}$ for the collection of parabolic subgroups of full support, whose core $I$ determines a subgroup $\langle I \rangle$ of parabolic type $[X]$, and which are  simple extensions of $\langle I\rangle$.
\end{definition}

\subsection{Orlik-Solomon exponents and the matrix $U$}
\label{ssec:OS-exps}

Much of the combinatorics of finite Coxeter groups $W$ is encoded in the reflection arrangement $\AAA_W$, the restricted arrangements $\AAA_W^X$ (for any flat $X\in\LLL_{\AAA_W}$) and their invariants. Orlik and Solomon \cite[Thm.~1.4]{OS_Cox_Arr} have shown that the characteristic polynomials of the arrangements $\AAA_W^X$ always factor as 
\begin{equation}
\chi(\AAA_W^X,t)=\prod_{i=1}^{\dim(X)}(t-b_i^X),\label{eq:defn:OS_exps}
\end{equation}
for \emph{positive, integer} numbers $b_i^X$. These numbers are known as the \defn{Orlik-Solomon exponents} for the parabolic type $[X]$; we write them in increasing order $b_i^X\leq b_{i+1}^X$ and we always have $b_1^X=1$ since $\AAA_W^X$ is central. In the case $X=V$, we get the whole arrangement $\AAA_W$ and the Orlik-Solomon exponents $b_i^V$ are the \emph{exponents of $W$} as discussed in the introduction. This factorization property comes down to the fact\footnote{Proven case-by-case for all real reflection groups in \cite{OT_A^X_is_free} and uniformly for Weyl groups in \cite{Douglass_A^X_free}.} that all restricted arrangements $\AAA_W^X$ are free (see \S\ref{ssec:free_arrts}). Combining \eqref{eq:defn:OS_exps} with Defn.~\ref{defn:beta(A)}, the beta invariants of the restricted arrangements are given by
\begin{equation}
\beta(\AAA_W^X)=\prod_{i=2}^{\dim(X)}(b_i^X-1).\label{eq:beta(A^X)}
\end{equation}

As an intermediate step in computing the characteristic polynomials $\chi(\AAA_W^X,t)$, Orlik and Solomon \cite[\S2]{OS_Cox_Arr} consider a square matrix $U$ whose rows and columns are indexed by parabolic types. Its entries are important for our Theorem~\ref{thm:main} so we recall the definition here.

\begin{definition}\label{defn:OS:U}
Let $W$ be an irreducible, finite Coxeter group, and  $[X],[Y]\in\LLL_{\AAA_W}/W$ two of its parabolic types. We define $u^{}_{[X],[Y]}$ to be the number of flats in the restricted arrangement $\AAA_W^X$, for any representative $X\in [X]$, that are of parabolic type $[Y]$. That is,
\[
\defn{u^{}_{[X],[Y]}}:=\#\Bigg\{ {Z\in [Y]\text{ such that }Z\subseteq X \atop \text{for a fixed } X\in[X]} \Bigg\}.
\]
We arrange these numbers in a square matrix, whose rows and columns are identically indexed by the parabolic types of $W$ in non-increasing dimension, and we call it the \defn{Orlik-Solomon matrix $U$}. For each irreducible Coxeter group, the Orlik-Solomon matrices are given in \cite[Appendix~C]{OT_book}.
\end{definition}

\subsection{Proof of the main theorem}
\label{ssec:main_thm}
We start by a few propositions that relate the Coxeter-theoretic objects of \S\ref{ssec:coxeter_stuff} with the combinatorial objects of \S\ref{ssec:definitions_far_nfar}.

\begin{lemma}[Geometric interpretation of support and core]\label{lem:geom_interpr_support}
Let $W$ be a finite Coxeter group with set of simple generators $S$, reflection arrangement $\AAA$, and fundamental chamber $C_0$. For any flat $X\in\LLL_{\AAA}$, let $I(X)$ and $J(X)$ denote respectively the core and support of the parabolic subgroup $W_X$. Then the following hold.
\begin{enumerate}[(i)]
    \item $C_0^{I(X)}$ is the \emph{minimal} face of $C_0$ whose span contains $X$, and
    \item $C_0^{J(X)}$ is the \emph{maximal} face of $C_0$ that is contained in $X$.
\end{enumerate}
\end{lemma}
\begin{proof}
For the first implication $(i)$, notice that for any subset $I\subseteq S$, we have that $\langle I\rangle\leq W_X$ if and only if $\spn(C_0^I)\supseteq X$. By its definition, the core $I(X)$ is the unique \emph{maximal} standard parabolic subgroup contained in $W_X$, so that dually $C_0^{I(X)}$ must be the minimal face whose span contains $X$.

For the second implication $(ii)$ start with any subset $J\subseteq S$; we have that $W_X\leq \langle J\rangle$ if and only if $\bigcap_{s\in J} V^s\subseteq X$ if and only if $C_0^J\subseteq X$. This implies (recall the definition of support below equation \eqref{eq:defn:I_and_J}) that the support of $W_X$ corresponds to the largest face of $C_0$ that is contained in $X$.
\end{proof}

The following corollary states that parabolic subgroups of full support correspond to faraway flats for the fundamental chamber (see the paragraph below Defn.~\ref{defn: far planes}).

\begin{corollary}\label{cor:suppor_interp}
Let $W$ be a finite Coxeter group with reflection arrangement $\AAA$ and fundamental chamber $C_0$. For any flat $X\in\LLL_{\AAA}$, the parabolic subgroup $W_X$ is full support if and only if $X$ is a faraway flat for $C_0$.
\end{corollary}
\begin{proof}
We need to show that $W_X$ is full support if and only if $X\cap C_0=\bm 0$. This is immediate after Lemma~\ref{lem:geom_interpr_support}:$(ii)$ since the maximal face of $C_0$ that is contained in $X$ is equal to $X\cap C_0$ and since $C_0^S=\bm 0$ (for the set of simple generators $S$ of $W$).
\end{proof}

In fact, we can do a bit more when the (full support) parabolic subgroups of $W$ are simple extensions of their cores. Recall from Defn.~\ref{defn:sets:g(I)y} that $\nfw(I)_{[Y]}$ consists of the parabolic subgroups of type $[Y]$ that are full support and have core $I$. If $\op{rank}(W_Y)=|I|+1$, then all groups in $\nfw(I)_{[Y]}$ are by construction simple extensions of their core (see the discussion preceding Defn.~\ref{defn:nfwse}).

\begin{corollary}\label{cor:g(I)_Y=nf_(A,Y)(C,C^I)}
Let $W$ be a finite Coxeter group with reflection arrangement $\AAA$, set of simple generators $S$ and associated fundamental chamber $C_0$. For any pair $(I,[Y])$ of a subset $I\subsetneq S$ and a parabolic type $[Y]\in\LLL_{\AAA}/W$ such that $|I|+1=\op{rank}(W_Y)$, the set $\nfw(I)_{[Y]}$ is in bijection with the collection of nearest faraway flats for $C_0$ that have associated face $C_0^I$ (recall Defn.~\ref{defn: associated_face}) and belong to the orbit $[Y]$. That is, we have 
\[
\nfw(I)_{[Y]}=\nfar_{\AAA,[Y]}(C_0,C_0^I).
\]
\end{corollary}
\begin{proof}
A parabolic subgroup $W_Z$ belongs to the set $\nfw(I)_{[Y]}$ if and only if it is full support, has core $I$ and the flat $Z$ lies in the orbit $[Y]$. By Lemma~\ref{lem:geom_interpr_support} and Corollary~\ref{cor:suppor_interp} this is equivalent to requiring that $Z$ is in $[Y]$, it is faraway for the chamber $C_0$, and that $C_0^I$ is the minimal face $F$ of $C_0$ such that $\spn(F)\supseteq X$. Since moreover we have assumed that $|I|+1=\op{rank}(W_Z)$, we must have that $\dim\big(\spn(C_0^I)\big)=\dim(Z)+1$. This means that $Z$ is in fact a \emph{nearest} farawat flat for $C_0$ with associated face $C_0^I$ and the proof is complete.
\end{proof}

We are now ready to give the main Theorem and its proof.
\mainthm
\begin{proof}
We start by applying Corollary~\ref{cor:g(I)_Y=nf_(A,Y)(C,C^I)} and translating the enumeration of full support parabolic subgroups with respect to their cores, to the enumeration of nearest faraway flats with respect to the associated faces. We will have that
\[
\Big|\nfw\big([X]\big)_{[Y]}\Big|=\sum_{I\subseteq S\atop V^I\in [X]}\Big|\nfw(I)_{[Y]}\Big|=\sum_{I\subseteq S\atop V^I\in[X]}\Big|\nfar_{\AAA,[Y]}\left(C_0,C_0^I\right)\Big|,
\]
where the first equality is just Definition~\ref{defn:sets:g(I)y}. The following lines give the enumeration of the corresponding (nearest) faraway flats; immediately below, we elaborate each of the steps.
\begin{align*}
\sum_{I\subseteq S\atop V^I\in[X]}\Big|\nfar_{\AAA,[Y]}\left(C_0,C_0^I\right)\Big|&=\dfrac{1}{|W|}\cdot \sum_{C\in\CCC(\AAA)}\sum_{F\subseteq C\atop \spn(F)\in[X]} \Big|\nfar_{\AAA,[Y]}(C,F)\Big|\\
&=\dfrac{1}{|W|}\sum_{Z\in[X]}|W_Z|\cdot \sum_{K\in\CCC(\AAA^Z)}\Big|\far_{\AAA^Z,\AAA^Z\cap [Y]}(K)\Big|\\
&=\dfrac{1}{|W|}\sum_{Z\in [X]}|W_Z|\cdot\left( 2\cdot \Big|\AAA^Z\cap [Y]\Big|\cdot \beta(\AAA^Z)\right)\\
&=[W:N(X)]\cdot \dfrac{|W_X|}{|W|}\cdot 2\cdot u^{}_{[X],[Y]}\cdot\beta(\AAA^X)=\dfrac{2\cdot u^{}_{[X],[Y]}}{[N(X):W_X]}\cdot\beta(\AAA^X).
\end{align*}
The first equation holds because the action of $W$ respects the intersection patterns of orbits of flats on chambers (and there are $|W|$-many chambers in $\CCC(\AAA)$). The second equation is a reordering of the summation (each flat $Z\in [X]$ is adjacent to $|W_Z|$-many chambers of $\CCC(\AAA)$ while a chamber $K$ of the restricted arrangement $\AAA^Z$ is a face in each of those $|W_Z|$-many chambers) followed by an application of Proposition~\ref{prop:nfar-to-far}. The third equation is an application of the (double counting) Lemma~\ref{lem: double_counting} with $\PPP=\AAA^Z\cap [Y]$. For the fourth equation, we have that $\big|[X]\big|=[W:N(X)]$ by definition of the group $N(X)$ and we have that $\big|\AAA^Z\cap[Y]\big|=u^{}_{[X],[Y]}$ by the definition of the Orlik-Solomon matrix $U$ (see Defn.~\ref{defn:OS:U}).

Finally, the last equation is just a cancellation of the two $|W|$ factors and the statement of the theorem follows after \eqref{eq:beta(A^X)}.
\end{proof}

\begin{remark}\label{rem:contant_ratio}
Notice that in Theorem~\ref{thm:main} the quantity $\dfrac{\big|\nfw([X])_{[Y]}\big|}{u_{[X],[Y]}}$ does not depend on the parabolic type $[Y]$. In fact, if we denote by \defn{$\nu_{[X]}$} the number of standard parabolic subgroups of $W$ that are of parabolic type $[X]$, we will have that 
\[
\dfrac{\big|\nfw([X])_{[Y]}\big|}{u_{[X],[Y]}}=\dfrac{2}{[N(X):W_X]}\cdot\prod_{i=2}^{\dim(X)}(b_i^X-1)=\nu_{[X]}\cdot\prod_{i=2}^n\dfrac{b_i^X-1}{b_i^X+1},
\]
which relies on the formula $\nu_{[X]}\cdot [N(X):W_X]=\prod_{i=1}^{\dim(X)}(b_i^X+1)$ due to Orlik and Solomon (see \cite[(4.2)]{OS_Cox_Arr}). This is no longer true if $\dim(X)\neq \dim(Y)+1$, see \S\ref{sssec:extra:not_simple_ext}.
\end{remark}

At this point we will give a few corollaries of our main Theorem~\ref{thm:main}, recovering in particular the formulas of Thiel and Chapoton. We start with \cite[Thm.~1.2]{thiel} slightly rephrased to model the formulas of \cite[Prop.~6.6:(2)]{sommers}.

\begin{corollary}\label{cor:chap_form_krew}
Let $W$ be a finite, irreducible Coxeter group of rank $n\geq 2$ and let $t$ be a reflection of $W$ with fixed hyperplane $H:=V^t$. The number $\big(f_W\big)_{[H]}$ of reflections of full support that are conjugate to $t$ is given by
\[
\big(f_W\big)_{[H]}=\dfrac{\prod_{i=1}^{n-1}(h-1-e_i)}{[N(H):W_H]},
\]
where $h:=e_n+1$ is the Coxeter number of $W$ and $e_1,\ldots,e_n$ are the exponents of $W$, and where $N(H)$ and $W_H$ are respectively the setwise and pointwise stabilizers of $H$.
\end{corollary}
\begin{proof}
A reflection $t$ of full support corresponds to a parabolic subgroup $W_H:=\{\op{id},t\}$ of full support where $H=V^t$. In this case, the maximum standard parabolic subgroup of $W$ contained in $W_H$ has to be the identity subgroup $\{\op{id}\}\leq W$ (since $t$ cannot be a simple reflection) which corresponds to an empty core $I=\emptyset$. The parabolic type of $\{\op{id}\}$ is then the orbit $[V]$ of the ambient space $V$ and we will have $u_{[V],[H]}=[W:N(H)]$. Since $b_i^V=e_i$, we will have by Theorem~\ref{thm:main}
\[
\big(f_W\big)_{[H]}=\Big|\nfw\big([V]\big)_{[H]}\Big|=\dfrac{2\cdot[W:N(H)]}{|W|}\cdot  \prod_{i=2}^n(e_i-1)=\dfrac{\prod_{i=1}^{n-1}(h-1-e_i)}{[N(H):W_H]},
\]
where in the last equality we used that $|W_H|=2$ and that $e_i=h-e_{n+1-i}$ (this is known as exponent duality, see for instance \cite[Lemma~3.16]{Humphreys}).
\end{proof}

In our next corollary we give the enumeration of parabolic subgroups $W'\leq W$ of full support that are simple extensions of some standard parabolic subgroup, but without keeping track of the parabolic type of $W'$.

\begin{corollary}\label{cor:nfwse:enum}
Let $W$ be an irreducible, finite Coxeter group and let $[X]$ be a parabolic type for $W$. Then, the set $\nfwse\big([X]\big)$ of parabolic subgroups of full support, that are simple extensions of some standard parabolic subgroup of type $[X]$, has size given by the formula
\[
\Big|\nfwse\big([X]\big)\Big|=\dfrac{2\cdot|\AAA^X|}{[N(X):W_X]}\cdot\prod_{i=2}^{\dim(X)}(b_i^X-1),
\]
where $|\AAA^X|$ is the number of hyperplanes in $\AAA^X$ and the remaining invariants are as in Thm.~\ref{thm:main}.
\end{corollary}
\begin{proof}
This is an immediate corollary of Theorem~\ref{thm:main}; we only use that $\nfwse\big([X]\big)$ is a disjoint union of the sets $\nfw\big([X]\big)_{[Y]}$ for which $\dim(X)=\dim(Y)+1$ and that 
\[
\sum_{[Y]\in\LLL_{\AAA}/W\atop \dim(Y)=\dim(X)-1}u^{}_{[X],[Y]}=|\AAA^X|,
\]
which is just Definition~\ref{defn:OS:U}.
\end{proof}

We now give a proof of Chapoton's formula (Proposition~\ref{Prop:intro:chap_form}) but as a corollary of Theorem~\ref{thm:main}.

\begin{corollary}\label{cor:chap_form}
The number $f_W$ of reflections of full support in a finite, irreducible Coxeter group $W$ of rank $n$ is given by the formula
\begin{equation}
f_W=\dfrac{nh}{|W|}\cdot\prod_{i=2}^n(e_i-1),\label{eq:main:chapoton_form}
\end{equation}
where $h:=e_n+1$ is the Coxeter number of $W$ and $e_1,\ldots,e_n$ are the exponents of $W$.
\end{corollary}
\begin{proof}
As in the proof of Corollary~\ref{cor:chap_form_krew} we have that $f_W=\Big|\nfwse\big([V]\big)\Big|$. The statement now follows after Corollary~\ref{cor:nfwse:enum} since $N(V)=W$, $W_V=\{\op{id}\}$, $b_i^V=e_i$, and by the formula $2|\AAA|=hn$ (see for instance \cite[Prop.~3.18]{Humphreys}).
\end{proof}

\subsection{Further questions}\label{ssec:extras}

We finish here with a short discussion of the natural extensions but also limitations of Theorem~\ref{thm:main}. In particular, we give in Table~\ref{table:H4:option2} the cardinalities, computed using SageMath \cite{sagemath}, of all sets $\nfw(I)_{[Y]}$ for the group $H_4$.

\subsubsection{When the group $W$ is reducible}
\label{sssec:extra:W_reduc}
To make the presentation of our results easier, we have always assumed so far that the finite Coxeter group $W$ is irreducible. This is not a real problem and the general case can easily be reduced to our Theorem~\ref{thm:main}. 

Consider for example a \emph{reducible} group $W$ whose decomposition into irreducibles is given by $W=W_1\times\cdots\times W_s$ and let $W',W''\leq W$  be two parabolic subgroups with $\op{rank}(W'')=\op{rank}(W')+1$. Their irreducible decompositions are of the form $W'=W'_1\times\cdots\times W'_s$ and $W''=W''_1\times \cdots\times W''_s$, where $W'_i$ and $W''_i$ are (possibly trivial) parabolic subgroups of $W_i$. Now, since $\op{rank}(W'')=\op{rank}(W')+1$, we must have that $W'_i=W''_i$ for all but one index $i=1,\ldots,s$; let's assume that this index is $i=1$, so that $\op{rank}(W''_1)=\op{rank}(W'_1)+1$. 

Now, if $W''$ is assumed full support, then all $W''_i$ must be full support in $W_i$ and if $W'$ is assumed to be standard, then all $W'_i$ must be standard in $W_i$. This forces that $W'_i=W''_i=W_i$ for $i\neq 0$ since $W'_i=W''_i$ must be both standard and full support parabolic subgroups of $W_i$. This means that we have reduced the counting of parabolic subgroups of full support in $W$ that are simple extensions of their core, to doing so in each of the irreducible components $W_i$ of $W$, which is exactly the setting of Theorem~\ref{thm:main}.

\subsubsection{When the parabolic subgroups are not full support}
\label{sssec:extra:not_full_sup}

After the definition of core and support in \eqref{eq:defn:I_and_J}, we have only addressed the case of parabolic subgroups whose support equals the whole set of simple reflections $S$. This is also a superficial restriction, since any parabolic subgroup $W'\leq W$ with support $J$ is full support in the standard parabolic subgroup $\langle J\rangle\leq W$. It may be that $\langle J\rangle$ is a \emph{reducible} Coxeter group but as we show in \S\ref{sssec:extra:W_reduc} our Theorem~\ref{thm:main} naturally extends to that case.

\subsubsection{When we want to keep track of the core explicitly, not just its parabolic type}\label{sssec:extra:I_vs_[V^I]}

Even though in \eqref{eq:defn:I_and_J} we defined the core of a subgroup $W'\leq W$ as a specific subset $I\subseteq S$, the Theorem~\ref{thm:main} only gives formulas for the cardinality of the sets $\nfwse([X])$ (and their refinements by type) where we only keep track of the parabolic type $[X]$ of the core $I$. It is natural to ask for a further refinement of the theorem and a formula for the size of the sets $\nfwse(I)$ themselves.

As the following Tables~\ref{table:H4:example},\ref{table:E8:example} suggest, a \emph{product} formula seems unlikely to exist; notice in particular the value $\nfwse(\{8\})=43$ for $E_8$. The entries of the tables were calculated via a computer using SageMath, while the Coxeter presentations given in the captions specify the choice of simple systems. We have given in bold the quantities that correspond to the results of Theorem~\ref{thm:main} and Corollary~\ref{cor:nfwse:enum} (in both groups $H_4$ and $E_8$ any two reflections are conjugate so that all singleton subsets of $S$ have the same parabolic type $[H]$). Notice that, as discussed in Remark~\ref{rem:contant_ratio}, the ratio $\Big|\nfwse\big([H]\big)_{[Y]}\Big|/u^{}_{[H],[Y]}$ is constant; it is equal to $3$ and $22/7$ respectively in $H_4$ and $E_8$.

\begin{table}[h]
\begin{tabular}{|l||c|c|c|c|c||r|}
\hline
    $u^{}_{[H],[Y]}$ & \diagbox{$[Y]$}{$\nfw\big(I\big)^{\phantom{3}}_{[Y]}$}{$I$} & $\big\{1\big\}$ & $\big\{2\big\}$ & $\big\{3\big\}$ & $\big\{4\big\}$ & $\nfw\big([H]\big)_{[Y]}$ \\\hline
    $\bm{15} \phantom{\Big(}$ & $A_1^2$ & $11$ & $12$ & $12$ & $10$ & $\bm{45}$ \\\hline 
    $\bm{10} \phantom{\Big(}$ & $A_2$   & $8$  & $7$  & $7$  & $8$  & $\bm{30}$ \\\hline 
    $\bm{6}  \phantom{\Big(}$ &$I_2(5)$ & $4$  & $4$  & $4$  & $6$  & $\bm{18}$ \\\hline \hline
    $\big|\AAA^H\big|=\bm{31} \phantom{\Big(}$ & $\nfwse\big(I\big)$ & $23$ & $23$ & $23$ & $24$ & $\nfwse\big([H]\big)=\bm{93}$ \\\hline 
\end{tabular}
\caption{$2$-dimensional nearest faraway flats for $H_4$:\quad
\includegraphics[scale=0.3]{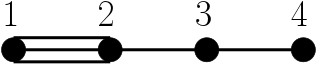}}
\label{table:H4:example}
\end{table}

\begin{table}[h]
\begin{tabular}{|l||c|c|c|c|c|c|c|c|c||r|}
\hline
    $u^{}_{[H],[Y]}$ & \diagbox{$[Y]$}{$\nfw\big(I\big)^{\phantom{3}}_{[Y]}$}{$I$} & $\big\{1\big\}$ & $\big\{2\big\}$ & $\big\{3\big\}$ & $\big\{4\big\}$ & $\big\{5\big\}$ & $\big\{6\big\}$ & $\big\{7\big\}$ & $\big\{8\big\}$ & $\nfw\big([H]\big)_{[Y]}$ \\\hline
    
    $\bm{63} \phantom{\Big(}$ & $A_1^2$ & $27$ & $25$ & $24$  & $24$ & $24$ & $24$  & $24$ & $26$ & $\bm{198}$ \\\hline
    
    $\bm{28}  \phantom{\Big(}$ &$A_2$ & $11$  & $10$  & $10$  & $10$ & $10$ & $10$ & $10$ & $17$ & $\bm{88}$ \\\hline \hline
    
    $\big|\AAA^H\big|=\bm{91} \phantom{\Big(}$ & $\nfwse\big(I\big)$ & $38$ & $35$ & $34$ & $34$ & $34$ & $34$ & $34$ & $43$ & $\nfwse\big([H]\big)=\bm{286}$ \\\hline 
\end{tabular}
\caption{$2$-dimensional nearest faraway flats for $E_8$:\ \  \includegraphics[scale=0.3]{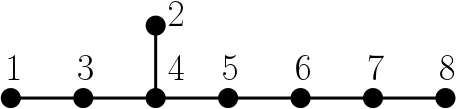} }
\label{table:E8:example}
\end{table}

\begin{remark}
In the coincidental \cite[\S3.1.5]{williams_thesis} types $A_n=S_{n+1}, B_n, I_2(m), H_3$ (also known as good reflection arrangements \cite[\S5.7]{aguiar_book}) all restricted arrangements $\AAA^X$ are combinatorially isomorphic to an irreducible reflection arrangement \cite[Thm.~5.28]{aguiar_book}, and their Orlik-Solomon exponents are the first $\dim(X)$-many exponents of $W$ \cite[\S3.3]{miller_foulkes}. Because of the first statement, the intersection patterns of orbits of flats and chambers are the same for all chambers of $\AAA^X$, and the cardinalities of the sets $\nfw(I)_{[Y]}$ with $\op{rank}(W_Y)=|I|+1$ only depend on $[Y]$ and the parabolic type of $\langle I\rangle$ (i.e., not on $I$ itself as in Tables~\ref{table:H4:example},\ref{table:E8:example}). In particular, if $W$ is a coincidental type of rank $n$ and with exponents $e_1,\ldots, e_n$, Theorem~\ref{thm:main} and the formula for $\nu_{[V^I]}$ from Remark~\ref{rem:contant_ratio} imply that
\[
\Big|\nfw(I)_{[Y]}\Big|=u^{}_{[V^I],[Y]}\cdot\prod_{i=2}^{n-|I|}\dfrac{e_i-1}{e_i+1}.
\]
\end{remark}

\subsubsection{When the parabolic subgroups are not necessarily simple extensions of their cores}\label{sssec:extra:not_simple_ext} The geometric techniques we developed in Section~\ref{sec:double_counting} seem to only be able to model the case of parabolic subgroups that are simple extensions of their cores. It is unclear how one might pursue studying the sets $\nfw(I)_{[Y]}$ without any further restrictions on the types $[Y]$. 

In Table~\ref{table:H4:option2} we give the complete picture for the group $H_4$ (presented as in Table~\ref{table:H4:example}) and its parabolic subgroups of full support. We observe again that product formulas for the cardinalities seem unlikely; for instance, there are $157$ parabolic subgroups of type $[A_2]$ that are full support and have core $I=\emptyset$. Moreover, it is clear that Remark~\ref{rem:contant_ratio} cannot be generalized to the case that the parabolic subgroups are not simple extensions of their cores. Again, the entries were calculated using SageMath and we have given in bold the quantities that can be derived by Theorem~\ref{thm:main} and Corollary~\ref{cor:nfwse:enum}. 

As a sanity check, we offer the following calculation. For each of the four types $[Y]$ of parabolic subgroups of rank $3$ (listed as $A_1\times A_2$, $A_1\times I_2(5)$, $A_3$, and $H_3$) the \emph{total number} of parabolic subgroups of type $[Y]$ that are full support must equal $u^{}_{[A_0],[Y]}-1$. To see this, note that the support of a rank-$3$ parabolic subgroup must have at least three elements; then either the subgroup is full support, or it is equal to (the group generated by) its support. This is easy to verify in Table~\ref{table:H4:option2} where we list the parabolic subgroups with respect to their cores. For example, there is a total of $299$ parabolic subgroups of type $[A_3]$ that are full support, which are made out of $197$ subgroups with core $I=\emptyset$, $87$ subgroups with core of type $[A_1]$, $5$ subgroups with core of type $[A_1^2]$, and $10$ subgroups with core of type $[A_2]$.

\begin{table}
\begin{tabular}{|c||c|ccccccccr|}
\hline
   $\Big[V^I\Big]$ & \diagbox{$I$}{$\nfw\big(I\big)^{\phantom{3}}_{[Y]}\phantom{3}$}{$[Y]$} & $A_1$ & $A_1^2$ & $A_2$ & $I_2(5)$ & $A_1\times A_2$ & $A_1\times I_2(5)$ & $A_3$ & $H_3$  & \multicolumn{1}{||r|}{$\nfwse(I)$} \\\hline\hline 

   \multirow{3}{*}{$A_0$} &$\phantom{\Big(} \emptyset \phantom{\Big)}$ & \multicolumn{1}{c|}{$42$} & $382$  & $157$  & $48$ & $457$ & $232$  & $197$ & $16$  & \multicolumn{1}{||r|}{$42$} \\\cline{2-10}

    &$\phantom{\Big(} \nfw\big([A_0]\big)_{[Y]} \phantom{\Big)}$ & \multicolumn{1}{c|}{$\bm{42}$} & $382$  & $157$ & $48$ & $457$ & $232$ & $197$ & $16$ & \multicolumn{1}{||r|}{$\nfwse([A_0])=\bm{42}$} \\\cline{2-10}
    
    &$\phantom{\Big(} u_{[A_0],[Y]} \phantom{\Big)}$ & \multicolumn{1}{c|}{$\bm{60}$} & $450$ & $200$ & $72$ & $600$ & $360$  & $300$ & $60$   & \multicolumn{1}{||r|}{$\big|\AAA\big|=\bm{60}$} \\\hline\hline

    \multirow{6}{*}{$A_1$} & $\phantom{\Big(} \big\{ 1\big\} \phantom{\Big)}$  &  & $11$ & $8$ &\multicolumn{1}{c|}{$4$} & $33$ & $24$ & $26$  & $7$ & \multicolumn{1}{||r|}{$23$}  \\\cline{2-10}
    
    &$\phantom{\Big(} \big\{ 2\big\} \phantom{\Big)}$ &  & $12$  & $7$ & \multicolumn{1}{c|}{$4$} & $33$ & $28$ & $23$ & $6$   & \multicolumn{1}{||r|}{$23$} \\\cline{2-10}
    
    &$\phantom{\Big(} \big\{ 3\big\} \phantom{\Big)}$ &  &  $12$ & $7$ & \multicolumn{1}{c|}{$4$} & $31$ & $32$ & $17$ & $8$   &\multicolumn{1}{||r|}{$23$} \\\cline{2-10}
    
    &$\phantom{\Big(} \big\{ 4\big\} \phantom{\Big)}$ &  & $10$ & $8$ &\multicolumn{1}{c|}{$6$} & $30$ & $29$ & $21$ & $8$   & \multicolumn{1}{||r|}{$24$} \\\cline{2-10}
    
    &$\phantom{\Big(} \nfw\big([A_1]\big)_{[Y]} \phantom{\Big)}$ &  & $\bm{45}$ & $\bm{30}$ & \multicolumn{1}{c|}{$\bm{18}$} & $127$ & $113$  & $87$ & $29$    & \multicolumn{1}{||r|}{$\nfwse([A_1])=\bm{93}$} \\\cline{2-10}
    
    &$\phantom{\Big(} u_{[A_1],[Y]} \phantom{\Big)}$ & $1$ & $\bm{15}$ & $\bm{10}$ & \multicolumn{1}{c|}{$\bm{6}$} & $40$ & $36$  & $30$ & $15$  & \multicolumn{1}{||r|}{$\big|\AAA^{A_1}\big|=\bm{31}$} \\\hline\hline 
    
    \multirow{5}{*}{$A^2_1$}&$\phantom{\Big(} \big\{ 1,3\big\} \phantom{\Big)}$ &  &  &  &  & $3$ & $4$ & $2$ & $1$   & \multicolumn{1}{||r|}{$10$} \\\cline{2-10}
        
    &$\phantom{\Big(} \big\{ 1,4\big\} \phantom{\Big)}$ &  &  &  &  & $3$ & $3$ & $2$ & $2$   & \multicolumn{1}{||r|}{$10$} \\\cline{2-10}
            
    &$\phantom{\Big(} \big\{ 2,4\big\} \phantom{\Big)}$ &  &  &  &  & $4$ & $3$ & $1$ & $2$    & \multicolumn{1}{||r|}{$10$} \\\cline{2-10}
    
    &$\phantom{\Big(} \nfw\big([A^2_1]\big)_{[Y]} \phantom{\Big)}$ &  & &  &  & $\bm{10}$ & $\bm{10}$ & $\bm{5}$ & $\bm{5}$   & \multicolumn{1}{||r|}{$\nfwse([A^2_1])=\bm{30}$} \\\cline{2-10}
    
    &$\phantom{\Big(} u_{[A^2_1],[Y]} \phantom{\Big)}$ &  & $1$ & $0$ & $0$ & $\bm{4}$ & $\bm{4}$  & $\bm{2}$ & $\bm{2}$   & \multicolumn{1}{||r|}{$\big|\AAA^{A^2_1}\big|=\bm{12}$} \\\hline\hline

    \multirow{4}{*}{$A_2$} &$\phantom{\Big(} \big\{ 2,3\big\} \phantom{\Big)}$ &  &  &  &  & $3$ & $0$ & $5$ & $2$  &   \multicolumn{1}{||r|}{$10$}  \\\cline{2-10}
    
    &$\phantom{\Big(} \big\{ 3,4\big\} \phantom{\Big)}$ &  &  &  &  & $2$ & $0$ & $5$ &  $3$ &  \multicolumn{1}{||r|}{$10$} \\\cline{2-10}
    
    &$\phantom{\Big(} \nfw\big([A_2]\big)_{[Y]} \phantom{\Big)}$ &  & &  &  & $\bm{5}$ & $\bm{0}$ & $\bm{10}$ & $\bm{5}$ &   \multicolumn{1}{||r|}{$\nfwse([A_2])=\bm{20}$} \\\cline{2-10}
    
    &$\phantom{\Big(} u_{[A_2],[Y]} \phantom{\Big)}$ &  &  & $1$ & $0$ & $\bm{3}$ & $\bm{0}$  & $\bm{6}$ & $\bm{3}$   & \multicolumn{1}{||r|}{$\big|\AAA^{A_2}\big|=\bm{12}$} \\\hline\hline 
        
    \multirow{3}{*}{$I_2(5)$} &$\phantom{\Big(} \big\{ 1,2\big\} \phantom{\Big)}$ &  &  &  &  & $0$ & $4$  & $0$ & $4$   & \multicolumn{1}{||r|}{$8$} \\\cline{2-10}

    &$\phantom{\Big(} \nfw\big([I_2(5)]\big)_{[Y]} \phantom{\Big)}$ &  & &  &  & $\bm{0}$ & $\bm{4}$  & $\bm{0}$ & $\bm{4}$    & \multicolumn{1}{||r|}{$\nfwse([I_2(5)])=\bm{8}$} \\\cline{2-10}
    
    &$\phantom{\Big(} u_{[I_2(5)],[Y]} \phantom{\Big)}$ &  &  &  & $1$ & $\bm{0}$ & $\bm{5}$  & $\bm{0}$ & $\bm{5}$  & \multicolumn{1}{||r|}{$\big|\AAA^{I_2(5)}\big|=\bm{10}$} \\\hline
    
\end{tabular}
\caption{Numbers of parabolic subgroups of full support in $H_4$, refined by their core (the maximal standard parabolic subgroup they contain) and parabolic type.}
\label{table:H4:option2}
\end{table}

\section{Acknowledgments} 
The main ideas for this work were developed in Colombia, during ECCO 2022 (Encuentro Colombiano de Combinatoria). I would like to heartily thank the organizers of the meeting and associated CIMPA Research School (Geometric methods in combinatorics) for their hospitality and the beautiful mathematical environment they created for the participants. I would also like to thank Philippe Nadeau for fascinating discussions during ECCO and for introducing me to BBC.

\bibliographystyle{alpha}
\bibliography{biblio}

\end{document}